\documentclass[11pt]{article}

\title{Smoothing of rational $m$-ropes}
\author{Edoardo Ballico\footnote{The first author was partially supported by MIUR and GNSAGA of INdAM (Italy).} \and Elizabeth Gasparim \and Thomas K\"oppe}
\date{}

\usepackage{amsmath,amsthm,amssymb}
\usepackage[margin=3cm]{geometry}

\newtheorem{theorem}{Theorem}[section]

\newtheorem{lemma}[theorem]{Lemma}

\theoremstyle{definition}
\newtheorem{definition}[theorem]{Definition}

\newtheorem{remark}[theorem]{Remark}

\DeclareMathOperator{\Ext}{Ext}
\DeclareMathOperator{\Aut}{Aut}
\DeclareMathOperator{\Pic}{Pic}
\DeclareMathOperator{\red}{red}
\DeclareMathOperator{\rank}{rk}
\DeclareMathOperator{\SExt}{\mathcal{E}\!\mathit{xt}}
\newcommand{\ce}{\mathrel{\mathop:}=}

\newcommand{\bbk}{{\mathbb{k}}}

\begin{document}

\maketitle

\begin{abstract}\noindent
In a recent paper, Gallego, Gonz\'{a}lez and Purnaprajna showed that
rational $3$-ropes can be smoothed.  We generalise their proof, and
obtain smoothability of rational $m$-ropes for $m \geq 3$.
\end{abstract}

\section{Introduction}\label{S1}

We generalise the smoothing theorem for rational $3$-ropes of Gallego,
Gonz\'alez and Purnaprajna to rational $m$-ropes with $m\geq 3$. Our
proof uses their construction presented in \cite{ggp}.

Let $C$ be a smooth, irreducible projective curve. A \emph{rope} $X$
of multiplicity $m \geq 2$ over $C$ is a nowhere reduced scheme $X$
whose reduced structure is $C$ and which locally looks like the first
infinitesimal neighbourhood of $C$ inside the total space of a vector
bundle of rank $m-1$ (\cite{c, ggp}).

Since the ideal sheaf $\mathcal{E} \ce \mathcal{I}_{C,X}$ of $C$
inside $X$ has square zero, it may be seen as a coherent
$\mathcal{O}_C$-sheaf, the so-called conormal bundle or conormal
module of $C$. As an $\mathcal {O}_C$-sheaf $\mathcal{E}$ is locally
free of rank $m-1$.

Our goal is to show smoothability of rational $m$-ropes. We recall
the precise definitions:

\begin{definition}
Let $Y$ be a reduced connected scheme and let $\mathcal{E}$ be a
locally free sheaf of rank $m-1$ on $Y$. A \emph{rope of multiplicity
$m$} or \emph{$m$-rope} on $Y$ with conormal bundle $\mathcal{E}$ is a
scheme $X$ with $X_{\red} = Y$ such that
\begin{itemize}
\item ${\mathcal I}^2_{Y,X} = 0 $ and 
\item ${\mathcal I}_{Y,X}\rvert_Y \cong \mathcal E$ as $\mathcal O_Y$-modules.
\end{itemize} 
\end{definition}

\begin{definition} A {\em smoothing of a rope} $X$ is a flat integral 
family $\mathcal X$ of schemes over a smooth affine curve $T$ such
that over a point $0 \in T$ we have $\mathcal X_0 = X$, and $X_t$ is a
smooth irreducible variety over the remaining points $t \in
T\setminus\{0\}$.
\end{definition}

Here we consider only the case when $Y = C$ is a smooth curve with
arithmetic genus $q \ce p_a(C) = 1 - \chi(\mathcal {O}_X)$, and we
work over an algebraically closed field of characteristic zero. Any
$m$-rope $X$ on $C$ with conormal module $\mathcal{E}$ gives an
extension class
\[ \epsilon \in \Ext^1_{\mathcal{O}_X}\bigl(\omega_C, \mathcal{E}\bigr)
   \cong H^1\bigl(C; \; \mathcal{E} \otimes \omega_C^*\bigr) \]
(cf.\ \cite[1.2]{g} or \cite[\S 1]{be} for the case $m=2$). Two ropes
$X$, $X'$ with conormal module $\mathcal{E}$ are isomorphic over $Y$
if and only if their extension classes are in the same orbit by the
action of $\Aut(\mathcal{E})$ on $\Ext^1_{\mathcal{O}_X} \bigl(
\omega_C, \mathcal{E}\bigr)$ (cf.\ \cite[1.2]{g}). There is an exact
sequence of $\mathcal{O}_X$-modules
\begin{equation}\label{eqa1}
  0 \longrightarrow \mathcal{E} \longrightarrow \mathcal{O}_X
  \longrightarrow \mathcal {O}_C \to 0 \text{ ,}
\end{equation}
and so $\chi(\mathcal{O}_X) = m \cdot \chi(\mathcal{O}_C) +
\deg(\mathcal{E}) = m(1-q) +\deg(\mathcal{E})$. Let $g \ce p_a(X) = 1
- \chi(\mathcal {O}_X)$ be the arithmetic genus of $X$. Obviously, if
$X$ is a flat limit of a family of smooth, connected projective
curves, then $\chi (\mathcal {O}_X) \leq 1$, i.e.\ $\deg(\mathcal{E})
\leq m(q-1)+1$, and $g$ is the genus of the nearby smooth
curves. Hence $g \geq 0$. Here (as in \cite[\S 4]{ggp}) we will only
consider \emph{rational} $m$-ropes, i.e.\ we will assume that
$q=0$. In the case $m=3$ Gallego, Gonz\'{a}lez and Purnaprajna proved
that if $p_a(X) \geq 0$, then the rational $3$-rope $X$ may be
smoothed, both as abstract scheme and as scheme embedded in a fixed
projective space \cite[Theorem 4.5]{ggp}. Here we use their proof to
solve the case $m \geq 4$, proving the following result.

\begin{theorem}[Main theorem]\label{i2}
Fix integers $r > m \geq 3$ and $g \geq 0$ and let $X$ be any rational
$m$-rope such that $1 - \chi(\mathcal {O}_X)=g$. Then $X$ may be
smoothed as an abstract scheme.

Moreover, there exist an embedding $j \colon X \to {\mathbb {P}}^r$
and a flat family $\{\mathcal{X}_t\}_{t\in T}$ of subschemes of
${\mathbb {P}}^r$ parametrised by an integral and smooth affine curve
$T$ with the following properties:
\begin{enumerate}
\item There exists a point $0 \in T$ such that $j(X) = \mathcal{X}_0$,
      and
\item for all $t \in T\backslash \{0\}$, $\mathcal{X}_t$ is a smooth
      connected curve of genus $g$ and degree $m \cdot \deg j(X)$.
\end{enumerate}
\end{theorem}

For the existence of embedding $j\colon X \to \mathbb{P}^r$ with a
fixed degree, see Lemma \ref{b8} and Remark \ref{b3}.

\section{Proof of the Main Theorem}\label{S3}

We begin by collecting a few results which show the existence of many
non-degenerate embeddings of an $m$-rope in $\mathbb{P}^r$ for all $r
\geq m+1$.

\begin{lemma}\label{b1}
Fix an integer $m \geq 3$ and let $E$ be a vector bundle of rank $m-1$
on $\mathbb{P}^1$. There is a uniquely determined sequence of integers
$b_1 \geq \dotsb \geq b_{m-1}$ such that $E \cong \bigoplus_{i=1}^{m-1}
\mathcal {O}_{\mathbb{P}^1}(b_i)$, and $\deg(E) = b_1 + \dotsb
+ b_{m-1}$. Then $E$ is rigid if and only if $b_1 \leq b_{m-1}+1$.
\end{lemma}
\begin{proof}
This is a classical result of the deformation theory of vector bundles
on $\mathbb{P}^1$.
\end{proof}

Let $X$ be an $m$-rope over $C$ with canonical module $\mathcal{E} =
\mathcal{I}_{C,X}$. Since $\mathcal{E}^2 = 0$ and $H^2(C; \;
\mathcal{E}) = 0$, there is an exact sequence of Abelian groups
\begin{equation}\label{eqe2}
  0 \longrightarrow H^1(C; \; \mathcal{E}) \longrightarrow
  \Pic(X) \longrightarrow \Pic(C) \longrightarrow 1 \text{ ,}
\end{equation}
in which the group structure of $H^1(C; \; \mathcal{E})$ as a subgroup
of $\Pic(X)$ is the usual addition of the $\bbk$-vector space (see
\cite[p.\ 446]{h0} for the case $\bbk = \mathbb{C}$, or the proof of
\cite[Proposition 4.1]{be} for an arbitrary field $\bbk$). Hence for
every $L \in \Pic(C)$ there exists an $L' \in \Pic(X)$ such that
$L'\rvert_C \cong L$. Now assume that $q \ce p_a(C) = 0$, so $C \cong
\mathbb{P}^1$. By Lemma \ref{b1} there are integers $a_1 \geq \dotsb
\geq a_{m-1}$ such that $\mathcal{E} \cong \bigoplus_{i=1}^{m-1}
\mathcal{O}_{\mathbb{P}^1}(-a_i)$.

\begin{lemma}\label{e0}
With the notation as above and still assuming $q=0$, fix $d \in
\mathbb{Z}$ and some $L_d \in \Pic(X)$ such that $L_d \rvert_C \cong
\mathcal{O}_{\mathbb {P}^1}(d)$. If $d \geq \max \{2,a_1+1\}$, then
$L_d$ is very ample, $h^1(X; \; L_d)=0$, and $h^0(X; \; L_d) = (m+1) d
- \sum_{i=1}^{m-1} a_i$.
\end{lemma}
\begin{proof}
The last assertion is obvious, because $h^1\bigl(\mathbb{P}^1; \;
\mathcal{O}_{\mathbb{P}^1}(d)\bigr) = h^1\bigl(\mathbb{P}^1; \;
\mathcal{E}(d)\bigr) = 0$. To check the very ampleness of $L_d$, it
suffices to prove that $h^0\bigl(X; \; \mathcal{I}_Z \otimes L_d\bigr)
= h^0\bigl(X; \; L_d) - 2$ (or, equivalently, $h^1\bigl(X; \;
\mathcal{I}_Z \otimes L_d\bigr) = 0$) for any length-$2$
zero-dimensional subscheme $Z \subset X$. Fix an affine neighbourhood
$U$ of $Z$ in $X$. Since every affine $m$-rope is split, there is a
retraction $u \colon U \to U \cap C$. There is a length-$2$
zero-dimensional scheme $W$ such that $Z \subset u^{-1}(W)$. Hence it
is sufficient to prove that $h^1\bigl(X; \; \mathcal{I}_{u^{-1}(W)}
\otimes L_d\bigr) = 0$. The latter vanishing is true, because
$h^1\bigl(\mathbb{P}^1; \; \mathcal{O}_{\mathbb{P}^1}(d-2)\bigr) =
h^1\bigl(\mathbb{P}^1; \; \mathcal{E}(d-2) \bigr) = 0$. Twisting
\eqref{eqa1} with $\mathcal{O}_{\mathbb{P}^1}(d)$ we get $h^1(X; \;
L_d) = 0$, and $h^0(X; \; L_d) = (m+1)d -\sum _{i=1}^{m-1} a_i$.
\end{proof}

Taking $d$ as in the proof of Theorem \ref{i2} below, we see that
in general we are able to smooth only certain types of embeddings.

\begin{lemma}\label{b6}
Let $Y$ be a smooth curve of genus $g$ and $m \in \mathbb{Z}$ such
that $m \geq \max \{g+1, 2\}$. Let $R$ be a general element in
$\Pic^m(Y)$.  There exists a general two-dimensional linear subspace
$V$ of $H^0(Y;\;R)$ that spans $R$, and any such $V$ determines a
degree-$m$ morphism $f \colon Y \to \mathbb{P}^1$. Then the sheaf $G
\ce f_* (\mathcal{O}_Y) \bigl/ \mathcal{O}_{\mathbb{P}^1}$ is locally
free of rank $m-1$, and $G$ is rigid.
\end{lemma}
\begin{proof}
Since $R$ is general and $m \geq g+1$, $h^1(Y;\;R)=0$. Thus, $h^0(Y;
\; R) = m + 1 - g \geq 2$ by Riemann-Roch. The generality of $R$
implies that $R$ is spanned, and hence a general two-dimensional
linear subspace $V$ of $H^0(Y; R)$ spans $R$. Any such $V$ determines
a degree-$m$ morphism $f \colon Y \to \mathbb{P}^1$. Since
$\mathcal{O}_Y$ is torsion-free, so is $f_*(\mathcal{O}_Y)$, which is
therefore locally free; also $h^0\bigl(\mathbb{P}^1; \;
f_*(\mathcal{O}_Y)\bigr) = h^0\bigl(Y; \; \mathcal{O}_Y\bigr) =
1$. Therefore $f_*(\mathcal{O}_Y)$ has precisely one trivial line
subbundle, so the sheaf $G \ce f_* (\mathcal{O}_Y) \bigl/
\mathcal{O}_{\mathbb{P}^1}$ is locally free. Let $b_1 \geq \dotsb \geq
b_{m-1}$ be the splitting type of $G$, so $b_1 + \dotsb + b_{m-1} =
\deg(G)$. Since $1-g = \chi(\mathcal{O}_Y) = \chi(G) +
\chi(\mathcal{O}_{\mathbb{P}^1}) = \deg(G) + m$, we get $\deg(G) =
1-m-g$. Since $h^0(Y; \; \mathcal{O}_Y) = h^0(\mathbb{P}^1; \;
\mathcal{O}_{\mathbb{P}^1}) = 1$, $h^0(\mathbb{P}^1; G) = 0$, i.e.\
$b_1 < 0$. Since $R \cong f^* (\mathcal{O}_{\mathbb{P}^1}(1))$, we
have $h^1\bigl(Y;\; R\bigr) = h^1\bigl(\mathbb{P}^1; \; G(1)\bigr)$ by
the projection formula. Since $h^1(Y; R) = 0$, we get $b_{m-1}+1 \geq
-1$. Hence $b_{m-1} \geq b_1 - 1$, and $G$ is rigid by Lemma \ref{b1}.
\end{proof}

\begin{lemma}\label{b7}
Fix integers $m, g$ such that $2 \leq m \leq g \leq 2m-2$ and let $Y$
be a general smooth curve with genus $g$. There exists a line bundle
$R \in \Pic^m(Y)$ such that $h^0(Y; \; R) = 2$ and $R$ is spanned.
Hence $R$ determines a degree-$m$ morphism $f \colon Y \to
\mathbb{P}^1$ such that $R \cong f^*\bigl( \mathcal{O}_{\mathbb{P}^1}
(1) \bigr)$. Then the sheaf $G \ce f_* (\mathcal{O}_Y) \bigl/
\mathcal{O}_{\mathbb{P}^1}$ is locally free of rank $m-1$, and $G$ is
rigid.
\end{lemma}
\begin{proof}
Brill-Noether theory gives the existence of $R \in \Pic^m(Y)$ such
that $h^0(Y;\;R)=2$ and $R$ is spanned \cite[Theorem V.1.1]{acgh}.
The sheaf $G$ is locally free by the same argument as in the proof of
Lemma \ref{b6}. Let $b_1 \geq \dotsb \geq b_{m-1}$ be the splitting
type of $G$. As in the proof of Lemma \ref{b6}, the projection formula
gives, for all $c \in \mathbb{Z}_{\geq 0}$,
\[ h^0\bigl(Y; \; R^{\otimes c}\bigr) = h^0\bigl(\mathbb{P}^1; \;
  \mathcal{O}_{\mathbb{P}^1}(c)\bigr) + h^0\bigl(\mathbb{P}^1; \;
  G(c)\bigr) \text{ , } \]
i.e.\ $h^0\bigl(\mathbb{P}^1; \; G(c)\bigr) = h^0\bigl(Y; \;
R^{\otimes c}\bigr) - c - 1$. Since $h^0(Y; \; R) = 2$, we get
$h^0\bigl(Y; \; G(1)\bigr) = 0$, i.e.\ $b_1 \leq -2$.  The
Gieseker-Petri theorem gives $h^1\bigl(Y; \; R^{\otimes 2}\bigr) = 0$
\cite[Cor. 5.7]{ac}. Hence $b_{m-1} \geq - 3$, and $G$ is rigid by
Lemma \ref{b1}.
\end{proof}

\begin{lemma}\label{b8.0}
Let $D \subset \mathbb {P}^r$ be a smooth rational curve of degree
$d>0$ and assume that $r \geq 2$. Let $\mathcal{N}_D$ be the normal
bundle of $D$ in $\mathbb{P}^r$ and $n_1\geq \dotsb \geq n_{r-1}$ its
splitting type. Then $n_{r-1} \geq d$.
\end{lemma}
\begin{proof}
The Euler sequence of $T\mathbb{P}^r$ shows that $T\mathbb{P}^r(-1)$
is spanned. Consequently, $T\mathbb{P}^r(-1)\rvert_D$ is spanned.
Since $D$ is a closed submanifold of $\mathbb{P}^r$, there is a
surjection $T\mathbb {P}^r\rvert_D \to \mathcal{N}_D$. Thus,
$\mathcal{N}_D(-1)$ is spanned, i.e.\ $n_{r-1}-d \geq 0$.
\end{proof}

\begin{lemma}\label{b8}
Fix integers $r > m \geq 2$, $d>0$, let $\mathcal{E}$ be a vector
bundle of rank $m-1$ on $\mathbb{P}^1$ of splitting type $e_1 \geq
\dotsb \geq e_{m-1}$, and let $X$ be the rational $m$-rope with conormal
bundle $\mathcal{E}$.

Fix any embedding $u \colon \mathbb{P}^1 \to \mathbb{P}^r$ (we do not
assume that $u(\mathbb{P}^1)$ spans $\mathbb{P}^r$) and let $d \ce
\deg u(\mathbb{P}^1)$. If $d \geq -e_{m-1}$, then there exists an
embedding $j \colon X \to \mathbb{P}^r$ such that $j\rvert_{X_{\red}}
= u$. Moreover, $h^1\bigl(\mathbb{P}^1; \; \mathcal{E} \otimes
u^*\mathcal{O}_{ \mathbb{P}^r}(1)\bigr) = 0$.
\end{lemma}
\begin{proof}
Set $D \ce u(\mathbb{P}^1)$ and let $\mathcal{N}_D$ be the normal
bundle of $D$ in $\mathbb{P}^r$. Let $n_1 \geq \dotsb \geq n_{r-1}$ be
the splitting type of $\mathcal{N}_D$; so the splitting type of
$\mathcal{N}^*_D$ is $-n_{r-1} \geq \dotsb \geq -n_1$. By Lemma
\ref{b8.0} we have $-n_i \leq -d$ for all $i$. By \cite[Proposition
2.1]{g} or \cite[Theorem 2.2]{ggp}, there is a one-to-one
correspondence between the surjections $\mathcal{N}_D^* \to
\mathcal{E}$ and embeddings $j \colon X \to \mathbb{P}^r$ such that
$j\rvert_{X_{\red}} = u$. Since $\rank \mathcal{N}_D^* = r-1 >
\rank\mathcal{E}$, a surjection $\mathcal{N}_D^* \to \mathcal{E}$
exists if $-d \leq e_{m-1}$, i.e.\ if $d \geq -e_{m-1}$. The last
sentence is obvious, because $h^1\bigl(\mathbb{P}^1; \;
\mathcal{E}(t)\bigr) = 0$ if and only if $t \geq -e_{m-1}-1$.
\end{proof}

As an aside, the following observation shows the existence of many
rational $m$-ropes in $\mathbb{P}^m$, but notice that their conormal
bundles must satisfy very strong restrictions. Since in the statement
of Theorem \ref{i2} we assume $r>m$, these are not the ropes that our
theorem addresses.

\begin{remark}[Embedding $m$-ropes in $\mathbb{P}^m$]\label{b3}
Fix integers $m \geq 2$, $d > 0$ and a vector bundle $\mathcal{E}$ on
$\mathbb {P}^1$ of rank $m-1$ with splitting type $e_1 \geq \dotsb
\geq e_{m-1}$. Let $X$ be the rational $m$-rope with conormal bundle
$\mathcal{E}$. Let $u \colon \mathbb{P}^1 \to \mathbb{P}^m$ be an
embedding such that the curve $D \ce u(\mathbb{P}^1)$ has degree $d$.
Let $\mathcal{N}_D$ be the normal bundle of $D$ in $\mathbb {P}^r$ and
$n_1 \geq \dotsb \geq n_{m-1}$ its splitting type. Since $\rank
\mathcal{E} = \rank \mathcal{N}_D$, any surjection $\mathcal{N}_D^*
\to \mathcal{E}$ must be an isomorphism.

Hence there exists an embedding $j \colon X \to \mathbb{P}^m$ such
that $j\rvert_{X_{\red}} = u$ if and only if $\mathcal{N}_D^* \cong
\mathcal{E}$ \cite[Proposition 2.1 (2)]{g}. Thus, the existence
problem of embeddings $j$ of $X$ such that $j\rvert_{X_{\red}}$ is
associated to a subseries of $H^0\bigl(\mathbb{P}^1; \;
\mathcal{O}_{\mathbb{P}^1}(d)\bigr)$, and is equivalent to the study
of all possible splitting types of the normal bundles $\mathcal{N}_D$
for some $D = u(\mathbb{P}^1)$.

The case $m=2$ is trivial, because we must have $1 \leq d \leq 2$, so
$D$ is either a line or a smooth conic. From now on we assume $m \geq
3$. We first consider the embeddings spanning $\mathbb{P}^m$. Thus we
assume for a moment $d \geq m$. If $m=3$, then the set of all
splitting types arising in this way is known, and the set of all
smooth rational space-curves with fixed normal bundle has a very
interesting geometry \cite{ev}. If $m>3$, then all possible splitting
types $n_1 \geq \dotsb \geq n_{m-1}$ that may arise if we allow the
map $\mathbb{P}^1 \to \mathbb{P}^r$ to be unramified but not
necessarily injective are described in \cite{s}. For arbitrary $m$,
the rigid vector bundle, i.e.\ the one with $b_{m-1} \geq b_1-1$,
arises as the normal bundle of the general degree-$d$ embedding
$\mathbb{P}^1 \hookrightarrow \mathbb{P}^m$.

Now we look at the embeddings for which $D$ spans a $k$-dimensional
linear subspace $M$ of $\mathbb{P}^m$ for some $k < m$. Let
$\mathcal{N}_{D,M}$ denote the normal bundle of $D$ in $M$ with
splitting type $b_1 \geq \dotsb \geq b_{k-1}$. By Lemma \ref{b8.0} we
have $b_{k-1} \geq d$. Since $\mathcal{N}_D \cong \mathcal{N}_{D,M}
\oplus \mathcal{O}_D(1)^{\oplus (m-k)}$, we get $n_i = b_i$ if $1 \leq
i \leq k-1$ and $n_i = d$ if $k \leq i \leq m-1$. Assume that
$\mathcal{E}^* \cong \mathcal{N}_D$, i.e.\ assume the existence of
degree-$d$ embedding $u$ of $\mathbb{P}^1$ and embedding $j \colon X
\to \mathbb{P}^m$ such that $j\rvert_{X_{\red}} = u$. By Lemma
\ref{b8.0} we have $e_1 \leq -d$. If $u(\mathbb{P}^1)$ spans
$\mathbb{P}^m$, then we have $e_1 \leq -d-1$. Since $\deg
\mathcal{N}_D = (m+1)d-2$, we have $e_1 + \dotsb + e_{m-1} = 2 -
(m+1)d$. According to \cite{s}, these are the only restrictions if we
allow unramified but non-injective maps $u$. Notice that
$h^1\bigl(\mathbb{P}^1; \; \mathcal{E} \otimes
u^*\mathcal{O}_{\mathbb{P}^r}(1)\bigr) = 0$ if and only if $d \geq
-e_{m-1} - 1$. If $u(\mathbb{P}^1)$ spans $\mathbb{P}^m$, then this
condition is satisfied only if $e_1 = e_{m-1}$, i.e.\ if and only if
$\mathcal{E}$ is balanced. \hfill $\bigl/\!\!\bigl/$
\end{remark}

We are now in a position to prove the main theorem.

\begin{proof}[Proof of Theorem \ref{i2}]
Let $\mathcal{E}$ be the conormal bundle of the $m$-rope $X$, and let
$e_1 \geq \dotsb \geq e_{m-1}$ be the splitting type of
$\mathcal{E}$. Also, let $G$ be the only rigid vector bundle on
$\mathbb{P}^1$ with rank $m-1$ and degree $1 - g - m$. If $g \geq m$,
then there exists a degree-$m$ covering $f \colon Y \to \mathbb{P}^1$
such that $Y$ is a smooth curve of genus $g$ and $f_*(\mathcal{O}_Y)
\cong \mathcal{O}_{\mathbb{P}^1} \oplus G$ \cite[Proposition 1]{b}.

There are various ways to see the equivalence between the rigidity of
$G$ and the statements in \cite{b} or in \cite[Proposition
2.1.1]{cm}. We can just use our Lemma \ref{b7}; alternatively the
reader may wish to consult \cite[2.4, 2.5]{s} or \cite[Remark
1]{b2}. For a different proof in the case $g \geq 2m+1$, see
\cite[Proposition 2.1.1]{cm}; there $Y$ is a general $m$-gonal curve
of genus $g$.

Lemma \ref{b6} gives the same result if $m \geq \max \{g+1,2\}$, and
Lemma \ref{b7} shows that such coverings are very common. Hence for
all integers $g \geq 0$ there is a smooth curve $Y$ of genus $g$ and a
degree-$m$ morphism $f \colon Y \to \mathbb{P}^1$ such that the
rank-$(m-1)$ vector bundle $f_*(\mathcal{O}_Y) \bigl/
\mathcal{O}_{\mathbb{P}^1}$ is rigid.

By \cite[Corollary 2.7]{ggp}, any rational $m$-rope $X'$ with conormal
bundle $G$ may be smoothed. Moreover, for any $r > m$, Lemma \ref{b8}
yields the existence of an embedding $j' \colon X' \hookrightarrow
\mathbb{P}^r$, where $j^{\prime*} \mathcal{O}_{\mathbb{P}^r}(1) = L'
\in \Pic(\mathbb{P}^1)$. (We also get the existence of such an
embedding for $r=m$ from Remark \ref{b3}.) Recall that
$H^1\bigl(X'_{\red} ; \; G \otimes L'\rvert_{X'_{\red}}\bigr) =
H^1\bigl(X'_{\red}; \; L'\rvert_{X'_{\red}}\bigr) = 0$. Then we may
apply \cite[Theorem 2.4]{ggp}, and $j'(X')$ can be smoothed inside
$\mathbb{P}^r$.

Note that since $G$ is rigid, the condition $h^1\bigl(\mathbb{P}^1 ;
\; G \otimes j'\rvert_{X'_{\red}}^*(\mathcal{O}_{\mathbb{P}^r}(1))\bigr)
= 0$ is satisfied if and only if $j'\rvert_{X'_{\red}}$ is a degree-$d$
embedding of $\mathbb{P}^1$ such that $d(m-1) + 1 - g - m \geq 1 - m$,
i.e.\ if and only if $d(m-1) \geq g$. Any vector bundle on
$\mathbb{P}^1$ of rank $m-1$ and degree $1-g-m$ is a degeneration of a
flat family of vector bundles on $\mathbb{P}^1$ isomorphic to $G$. To
get an embedding of $X$, we need a degree-$d$ embedding of
$\mathbb{P}^1$ with $d \geq -e_{m-1}$. So in principle we need to assume
that $\deg j(X_{\red}) \geq -e_{m-1}-1$. Note that we cannot fix the
same integer $\deg j(X_{\red})$ for all bundles with rank $m-1$ and
degree $1-g-m$. Very nice degeneration techniques are given in
\cite[Propositions 4.3 and 4.4]{ggp}, and Case~2 of the proof of
Theorem 4.5 shows that the smoothing (both as abstract schemes and as
embedded schemes) is true for arbitrary rational $m$-ropes with the
same arithmetic genus $g$, if we only consider as their supports
embeddings $u \colon \mathbb{P}^1 \to \mathbb{P}^r$ such that $\deg
u(\mathbb{P}^1) \geq - e_{m-1} -1$.

For reader's sake we summarise the part of the proof in \cite{ggp}
that we need: Let $X$ be a rational $m$-rope with arithmetic genus
$g$. Since $\mathcal{E}$ is a degeneration of $G$, there are an
integral scheme $S$, $0 \in S$, and a rank-$(m-1)$ vector bundle
$\mathbb{E}$ on $\mathbb{P}^1 \times S$ such that
\[ \mathbb{E}\rvert_{p_2^{-1}(0)} \cong \mathcal{E} \quad\text{and}\quad
   \mathbb{E}\rvert_{p_2^{-1}(s)} \cong G \text{ \ for a general $s \in S\setminus\{0\}$,} \]
where $p_2 \colon \mathbb{P}^1 \times S \to S$ is the projection onto
the second factor. Let $p_1 \colon \mathbb{P}^1 \times S \to
\mathbb{P}^1$ be the projection on the first factor. Set $\mathcal{A}
\ce \SExt^1_{p_2}\bigl(p_1^*(\omega_{\mathbb{P}^1}),
\mathbb{E}\bigr)$, where $\SExt^1_{p_2}$ is the relative
$\SExt^1$-sheaf with respect to $p_2$. The $\mathcal{O}_S$-sheaf
$\mathcal{A}$ is coherent. If the splitting type of $\mathcal{E}$ is
very unbalanced (i.e.\ if it contains an integer $\geq -2$), then
$\mathcal{A}$ is not locally free. The total space
$\mathbb{V}(\mathcal{A})$ of $\mathcal{A}$ parametrises a family of
rational $m$-ropes containing all $m$-ropes with conormal bundle $E$
and all $m$-ropes with conormal bundle $G$. Note that ``smoothing'' is
a closed condition. When $\mathcal{A}$ is locally free, then $\mathbb
{V}(\mathcal{A})$ is irreducible and the rope $X$ is smoothable. To
handle the general case, Gallego, Gonz\'{a}lez and Purnaprajna made
the following nice observation \cite[Proof of Proposition 4.4]{ggp}:
Even if $\mathcal{A}$ is not locally free, the fact that the fibres of
$p_2$ have dimension $1$ gives that $R^1 (p_2)_* \mathcal{A} =
H^1\bigl(\mathbb{P}^1; \; \mathcal{A}\rvert_{\{s\}}\bigr)$ for every
$s \in S$ \cite[II.5, Corollary 3]{m}. Since $S$ is affine, Theorem~A
of Serre gives the existence of $z \in H^0\bigl(S; \; R^1(p_2)_*
\mathcal{A})$ such that $z(0) = \epsilon$. The family of pairs
$\bigl\{\bigl(\mathbb{E}\rvert_{\{s\}}, z(s)\bigr)\bigr\}_{s\in S}$
gives a flat family of ropes with $X$ as fibre over $0$ and with
smoothable general fibre. Following
the arguments of Part~2 of the proof of \cite[Proposition 4.4]{ggp},
we see that $j$ and $j'$ can be fitted into a family of embeddings,
since the assumption that $\deg u(\mathbb{P}^1) \geq -e_{m-1}-1$ and
Lemma \ref{b8.0} imply the vanishing of $H^1\bigl(u(\mathbb{P}^1); \;
\mathcal{N}_{u(\mathbb{P}^1),\mathbb{P}^r} \otimes
\mathbb{E}\rvert_{\{s\}}\bigr)$.
\end{proof}

\noindent
Edoardo Ballico\\
Department of Mathematics\\
University of Trento\\
38050 Povo (TN), Italy\\
ballico@science.unitn.it

\bigskip\noindent
Elizabeth Gasparim, Thomas K\"oppe\\
School of Mathematics\\
University of Edinburgh\\
The King's Buildings, Edinburgh\\
Scotland EH9 3JZ, United Kingdom\\
Elizabeth.Gasparim@ed.ac.uk, t.koeppe@ed.ac.uk


\begin{thebibliography}{WWWW}

\bibitem{ac} E.\ Arbarello and M.\ Cornalba, \emph{Su una congettura di Petri},
  Comm.\ Math.\ Helv. \textbf{56} (1981), no.~1, 1--38.

\bibitem{acgh} E.\ Arbarello, M.\ Cornalba, P.\ A.\ Griffiths and J.\ Harris,
  Geometry of algebraic curves. I, Springer, Berlin, 1985.

\bibitem{b} E.\ Ballico, \emph{A remark on linear series of general $k$-gonal curves},
  Boll.\ UMI (7), \textbf{3--A} (1989), no.~2, 195--197.

\bibitem{b2} E.\ Ballico, \emph{Scrollar invariants of smooth projective curves},
  J.\ Pure Appl.\ Algebra \textbf{166} (2003), no.~3, 239--246.

\bibitem{be} D.\ Bayer and D.\ Eisenbud, \emph{Ribbons and their canonical embeddings},
  Trans.\ Amer.\ Math.\ Soc. \textbf{347} (1995), no.~3, 719--756.

\bibitem{c} K.\ A.\ Chandler, \emph{Geometry of dots and ropes}, Trans.\ Amer.\ Math.\
  Soc. \textbf{347} (1995), no.~3, 767--784.

\bibitem{cm} M.\ Coppens and G.\ Martens, \emph{Linear series on a general $k$-gonal curve},
  Abh.\ Math.\ Sem.\ Univ.\ Hamburg \textbf{69} (1999), 347-371.

\bibitem{ev} D.\ Eisenbud and A.\ Van de Ven, \emph{On the normal bundles of smooth rational space curves},
  Math.\ Ann. \textbf{256} (1981), no.~4, 453--463.

\bibitem{ggp} F.\ J.\ Gallego, M.\ Gonz\'{a}lez and B.\ P.\ Purnaprajna,
  \emph{Deformation of finite morphisms and smoothing of ropes}, Compositio Math.
  \textbf{144} (2008), no.~3, 673--688.

\bibitem{g} M.\ Gonz\'{a}lez, \emph{Smoothing of ribbons over curves}, J.\ Reine Angew.\
  Math. \textbf{591} (2006), 201--213.

\bibitem{h0} R.\ Hartshorne, \emph{Cohomological dimension of algebraic varieties},
  Ann.\ of Math. \textbf{88} (1968), no.~3, 402--450.

\bibitem{m} D.\ Mumford, Abelian varieties. Tata Institute of Fundamental Research
  Studies in Mathematics, No.~5, Oxford University Press, London 1970.

\bibitem{s} G.\ Sacchiero, \emph{Normal bundles of rational curves in projective space},
  Ann.\ Univ.\ Ferrara Sez.\ VII (N.\ S.) \textbf{26} (1980), 33--40.

\bibitem{sc} F.-O.\ Schreyer, \emph{Syzygies of canonical curves and special linear series},
  Math.\ Ann. \textbf{275} (1986), no.~1, 105--137.

\end{thebibliography}
\end{document}